\documentclass[12pt, a4paper]{article}
	
\usepackage{amsmath, amsthm, amscd, amsfonts, amssymb, graphicx, color}
\usepackage[bookmarksnumbered, plainpages]{hyperref}
\numberwithin{equation}{section}

 \newtheorem{thm}{Theorem}[section]
\newtheorem{cor}[thm]{Corollary}

\numberwithin{equation}{section}
\begin{document}
 \centerline{\Large{\bf Enumeration for the total number of all}}
\centerline{}
\centerline{\Large{\bf spanning forests of complete tripartite}}
 \centerline{}
\centerline{\Large{\bf graph based on the combinatorial}}
 \centerline{}
\centerline{\Large{\bf decomposition}}
 \centerline{}
\centerline{Sung Sik  U}
 \centerline{}
 \small \centerline{Faculty of Mathematics, \textbf{Kim Il Sung} University, D.P.R Korea}
\small \centerline{e-mail address : usungsik@yahoo.com}
\centerline{}
\centerline{}
\begin{abstract}
This paper discusses the enumeration for the total
  number of all rooted spanning forests of the labeled  complete tripartite graph.
  We enumerate the total number by a combinatorial decomposition.
\end{abstract}
%
%
%
%
\section{Introduction}
Y. Jin and C. Liu have enumerated the number of spanning forests of the labeled complete
 bipartite $K_{m,n}$ on $m$ and $n$ vertices by combinatorial method and by using the 
exponential generating function respectively (\cite{Ji1},\cite{Ji2}). And D. Stark \cite{Sta} 
has found the asymptotic number of labeled spanning forests of the complete bipartite graph $K_{m,n}$
as $m \to \infty$ when $m \leq n$  and $n=o(m^{6/5})$. L. A. Szekely \cite{Sze} gave a simple 
proof to the formula in \cite{Ji1} and a generalization for complete multipartite graphs. 
In \cite{Ege}, \cite{Lew} a bijective proof of the enumeration of spanning trees of the complete 
tripartite graphs and the complete multipartite graphs has been given respectively.  

Let $H_{m}$,$H_{n}$,$H_{p}$ denote the three disjoined vertex sets
of the complete tripartite $K_{m,n,p}$, that is, $K_{m,n,p}=(H_{m}, H_{n}, H_{p})$.
Let $V(K_{m,n,p})=H_{m}\bigcup H_{n}\bigcup H_{p}$ denote the vertex set
of $K_{m,n,p}$. The out-degree of a vertex $z$ will be denoted by $d^{+}(z)$,
while the in-degree of $z$ will be denoted by $d^{-}(z)$. 
Let $V(G)$ denote the 
vertex set of graph $G$. The goal of this paper is to give a closed formula of the 
enumeration for the total number of all spanning trees, forests of
the labeled complete tripartite $K_{m,n,p}$ by the combinatorial method. Throughout 
this paper, we will consider only the labeled graphs.

%
%
%
%

\section{Counting the number of spanning trees and forests of a labeled complete tripartite graph $K_{m,n,p}$}
Let $T(m,n,p)$ denote the set of all labeled spanning trees of the complete tripartite graph $K_{m,n,p}$.
Each tree $T$ in $K_{m,n,p}$ gives rise to labeled directed spanning tree $T^{'}$ with $z$ as a root, and all edges are directed to towards $z$.
Let $D(m,0;n,0;p,|\{z_{1}\}|)$ denote the set of all such directed trees with $z_{1}\subset H_{p}$ as a root. Clearly,
  \begin{equation*}
  |T_{m,n,p}|=|D(m,0;n,0;p,|\{z_{1}\}|)|.
  \end{equation*}

 For any $T \in D(m,0;n,0;p,|\{z_{1}\}|)$,
 
 \begin{equation*}
  d^{+}(z_{1})=0, d^{+}(z)=1, z\in V(K_{m,n,p})\backslash \{z_{1}\}.
  \end{equation*}

  It is well known \cite{Ege} that the number $f(m,l;n,k)$ of labeled spanning forests of $K_{m,n}=(H_{m}, H_{n})$, where in the forest every tree is
   rooted, there are  $l$ roots in $H_{m}$, $k$ roots in $H_{n}$, and the tree in
   the forest are not ordered, is equal to
  \begin{equation*}
  f(m,l;n,k)={m\choose l}{n\choose k}n^{m-l-1}m^{n-k-1}(km+ln-lk).
  \end{equation*}

 Our  proof  is based on the following combinatorial decomposition.
  Given a rooted spanning tree of the complete tripartite graph $K_{m,n,p}$ where the root is in $H_{p}$, we remove the root
  vertex from the tree to obtain a spanning forest of  the another tripartite graph.
  The roots of trees in this forest are in $H_{m}$ or $H_{n}$.

  \begin{thm}
  The number $|T_{m,n,p}|$ of labeled spanning trees of the complete tripartite graph $K_{m,n,p}$ is as follows:
  \begin{equation*}
   |T_{m,n,p}|=(m+n)^{p-1}(m+p)^{n-1}(n+p)^{m-1}(m+n+p).
  \end{equation*}
  \end{thm}

\begin{proof}

  We observe that a directed subgraph of  $K_{m,n,p}$ belongs to $D(m,0;n,0;p,|\{z_{1}\}|)$ if and only if, in the subgraph,
  \begin{equation*}
   d^{+}(z_{1})=0, d^{+}(z)=1, z\in V(K_{m,n,p})\backslash \{z_{1}\}
  \end{equation*}
  and the subgraph is (weakly)connected.
  Let $D(m,l;n,k)$ denote the set of all spanning forests of
  complete bipartite graph $K_{m,n}$, with $l$  roots in $H_{m}$ and $k$  roots in $H_{n}$, that is,
  \begin{equation*}
   f(m,l;n,k)=|D(m,l;n,k)|.
  \end{equation*}
  Let $F$ belongs to $D(m,l;n,k).$
 From $F$, we will construct the rooted spanning forests of  $K_{m,n,p}$ with root $z_{1}\in H_{p}$ as follows.
  First, link an edge $(z,v)$  between every $z\in H_{p}\backslash \{z_{1}\}$  and
  some $v\in V(F)$ (where  $V(F)$ denotes the vertex set of graph $F$).

   There are $(m+n)^{p-1}$ ways. Notice that the obtained graph $G$ has $k+l$(weakly)
  connected components each of which has a unique vertex in $H_{m}\bigcup H_{n}$ of out-degree zero.
  Now, for any fixed integer $t$, let $G'$ denote a graph obtained
  by adding $t$ edges consecutively to $G$ as follows.

  At each step we add an edge of the form $(a,b)$  where $b$  is any vertex of  $H_{p}\backslash \{z_{1}\}$
  and $a\in H_{m}\bigcup H_{n}$  is a  vertex of out-degree zero in any component not containing $b$  in the graph
   already constructed.

   The number of components decreases by one each
   time such an edge is added.

   Since $|H_{p}\backslash \{z_{1}\}|=p-1$ and  the number  of components not containing $b$  in the
  graph $G$  already constructed is $l+k-1$, there are $(l-1)(l+k-1)$ choices for
   the first such edge. Similarly, there are $(l-1)(l+k-2)$  choices for the second edge and in general 
$(l-1)(l+k-t)$  choices for the  $t$th edge, where, $0\leq t\leq l+k-1$, because the number of components in the graph $G$
   is $l+k$. The graph $G'$ constructed like this has $l+k-t$ components each of which has a unique vertex in $H_{m}\bigcup H_{n}$ of out-degree zero
   and the remaining vertices all have out-degree; if we add edges from these vertices of out-degree zero to $z_{1}$, we obtain a tree $T'$ in $D(m,0;n,0;p,|\{z_{1}\}|)$
   that contains $G$ and in which $d^{-}(z_{1})=l+k-t$. The order in which the $t$  edges
     are added to $G$ to form $G'$ is immaterial, so it follows that there are
  \begin{equation*}
 \frac{(p-1)(l+k-1)(p-1)(l+k-2)\cdots (p-1)(l+k-t)}{t!} = {l+k-1 \choose  t}(p-1)^{t}
  \end{equation*}
   rooted spanning trees $T'$ for fixed integer $t$.

  This implies that there are
    \begin{equation*}
    \sum_{t=0}^{l+k-1}{l+k-1 \choose t}(p-1)^{t}=p^{l+k-1} \\
  \end{equation*}
   spanning trees $T$ in $D(m,0;n,0;p,|\{z_{1}\}|)$ that contain $G.$

   Hence
   \begin{eqnarray*}
 |T_{m,n,p}|&=&|D(m,0;n,0;p,|\{z_{1}\}|)|
 =\sum_{l=0}^{m}\sum_{k=0}^{n}f(m,l;n,k)(m+n)^{p-1}p^{l+k-1}\\
 &=&\sum_{l=0}^{m}\sum_{k=0}^{n}{m \choose l}{n\choose  k}n^{m-l-1}m^{n-k-1}(km+nl-kl)(m+n)^{p-1}p^{l+k-1}\\
&=&(m+n)^{p-1}(m+p)^{n-1}(n+p)^{m-1}(m+n+p).
\end{eqnarray*}
Therefore, we get the required result. 
\end{proof}

\begin{cor}
  The number $f(m,0;n,0;p,1)$ of the labeled spanning trees of  $K_{m,n,p}$ with
   a root in $H_{p}$ as follows:
  \begin{equation*}
  f(m,0;n,0;p,1)=p(m+n)^{p-1}(m+p)^{n-1}(n+p)^{m-1}(m+n+p).
  \end{equation*}
  \end{cor}

  Let $D(m,0;n,0;p,|\{z_{i_{1}},z_{i_{2}},\cdots ,z_{i_{r}}\}|)$ be
  the set of the spanning forests of $K_{m,n,p}$ with roots $z_{i_{1}},z_{i_{2}},\cdots
  ,z_{i_{r}}$ in $H_{p}$.

  \begin{thm}
  The number $f(m,0;n,0;p,r)$ of the labeled spanning forests of the complete tripartite graph $K_{m,n,p}$ with $r$ roots in $H_{p}$ is as follows:
  \begin{equation*}
   f(m,0;n,0;p,r)={p \choose r}r(m+n)^{p-r}(m+p)^{n-1}(n+p)^{m-1}(m+n+p).
  \end{equation*}
  \end{thm}

\begin{proof}

  Let $z_{i_{1}},z_{i_{ in 2}},\cdots ,z_{i_{r}}$ in $H_{p}$ be vertices given as roots, $Z'=H_{p}\backslash \{z_{i_{1}},z_{i_{2}},\cdots
  ,z_{i_{r}}\}$
  and  $F$ belongs to $D(m,l;n,k)$.
  There are ${p \choose r}$ ways to choose the $r$ root in $H_{p}$.
  As in theorem 2.1, link an edge $(z,v)$ between every $z\in Z'$
  and  some $v\in V(F)$.
  There are $(m+n)^{p-r}$ ways.
  
Notice that the obtained graph $G$ has $k+l$ (weakly)connected components each of which has a unique vertex in $H_{m}\bigcup H_{n}$ of out-degree zero.
  As in the proof of theorem 2.1, for any fixed integer $t$ such that $0\leq t\leq l+k-1$,
  link an edge $(v,z)$ between any $z\in Z'$ and  a vertex $v\in H_{m}\bigcup H_{n}$ of out-degree zero in any component not containing $z$ in the
  graph already constructed, we repeat this procedure $t$ times.
  There are
  \begin{equation*}
 \frac{(p-1)(l+k-1)(p-1)(l+k-2)\cdots (p-1)(l+k-t)}{t!} = {l+k-1 \choose  t}(p-1)^{t}
  \end{equation*}
   rooted spanning forests $F'$.
   The every forests $F'$ thus obtained has $l+k-t$ (weakly)connected components each of which has a unique vertex in $H_{m}\bigcup H_{n}$ of out-degree
   zero.
   The number of the ways linking edges from $l+k-t$ vertices of
   out-degree zero in these components to $r$ vertices $z_{i_{1}},z_{i_{2}},\cdots
   ,z_{i_{r}}$ in $H_{p}\backslash Z'$ is equals to $r^{l+k-t}$.

  The number of the spanning forests with $r$ roots $z_{i_{1}},z_{i_{2}},\cdots
   ,z_{i_{r}}$ in $H_{p}$ of $K_{m,n,p}$ obtained from $F$ is as
   follows:
    \begin{equation*}
    \sum_{t=0}^{l+k-1}{l+k-1 \choose t}(p-1)^{t}r^{l+k-t}=p^{l+k-1}r. \\
  \end{equation*}
Hence
   \begin{eqnarray*}
 &&|D(m,0;n,0;p,|\{z_{i_{1}},z_{i_{2}},\cdots ,z_{i_{r}}\}|)|
 =\sum_{l=0}^{m}\sum_{k=0}^{n}f(m,l;n,k)(m+n)^{p-r}p^{l+k-1}r\\
 && \qquad = \sum_{l=0}^{m}\sum_{k=0}^{n}{m \choose l}{n\choose  k}n^{m-l-1}m^{n-k-1}(km+nl-kl)(m+n)^{p-r}p^{l+k-1}r\\
 && \qquad = (m+n)^{1-r}r|D(m,0;n,0;p,|\{z_{1}\}|)|\\
 && \qquad = r(m+n)^{p-r}(m+p)^{n-1}(n+p)^{m-1}(m+n+p).
 \end{eqnarray*}
Therefore,
  \begin{eqnarray*}
 f(m,0;n,0;p,r)
 &=&{p \choose r}|D(m,0;n,0;p,|\{z_{i_{1}},z_{i_{2}},\cdots
 ,z_{i_{r}}\}|)|\\
 &=&{p \choose r}r(m+n)^{p-r}(m+p)^{n-1}(n+p)^{m-1}(m+n+p).
 \end{eqnarray*}
\end{proof}

%
%

\section{Counting the total number of all spanning forests  of $K_{m,n,p}$}

\begin{thm}
  \label{T31}
  The total number $S(m,n,p)$ of all spanning forests  of $K_{m,n,p}$ is as follows:
 \begin{equation*}
S(m,n,p)=(m+n+1)^{p-1}(m+p+1)^{n-1}(n+p+1)^{m-1}(m+n+p+1)^{2}.
 \end{equation*}
 \end{thm}

\begin{proof}

  Let  $B(p,r)$ denote  the set of  spanning forests of the complete tripartite graph $K_{m,n,p}$ in which
    $r$  roots are in $H_{p}$ and remain roots are in $H_{m}$  or  $H_{n}$. Let $F$ belongs to $D(n,l;n,k)$.
  From $F$, we will construct the rooted spanning forests of  $K_{m,n,p}$  with $r$ roots  in $H_{p}$ as follows.
  Let $z_{i_{1}},z_{i_{2}},\cdots ,z_{i_{r}}\in H_{p}$   be  vertices given as roots. The number of
   ways which select  $r$  roots in $H_{n}$  is equal to  ${p \choose r}$.
  Let $Z'=H_{p}\backslash \{z_{i_{1}},z_{i_{2}},\cdots ,z_{i_{r}}\}$.
 
 First, link an edge $(z,v)$  between every $z\in Z'$  and
   some $v\in V(F)$ (i.e., vertex $v$  of  forest $F$).
 There are $(m+p)^{p-r}$ ways. Notice that the obtained graph $G$ has $k+l+r$
   (including components consisting of $z_{i_{1}},z_{i_{2}},\cdots ,z_{i_{r}}\in H_{p}$) weakly
  connected components.

  Let $t$ denote any fixed integer such that $0\leq t \leq l+k-1$,
  $H$ denote a graph obtained by adding $t$ edges consecutively  to $G$ as follows.
  At each step we add an edge of the form $(a, b)$ where $b$ is any
  vertex of $Z'$ and $a\in H_{m}\bigcup H_{n}$ is a root of any
  component not containing $b$ in the graph already constructed. The
  number of components decreases by one each time such an edge is added.
  Since $|Z'|=p-r$ and the number of components not containing $b$
  in the graph $G$ already constructed is $l+k-1$, there are
  $(p-r)(l+k-1)$ choices for the first such edge, $(p-r)(l+k-2)$ choices for the second edge,$\cdots $
  , and $(p-r)(l+k-t)$ choices for the $t$th edge.
  The order in which the $t$ edges are added to $G$ to form $H$ is
  immaterial, so it follows that there are
   \begin{equation*}
  \frac{(p-r)(l+r-1)(p-r)(l+r-2)\cdots (p-r)(l+r-t)}{ t!}={l+r-1 \choose t}(p-r)^{t}
  \end{equation*}
  ways.

  The graph $H$ constructed like this has $l+k-t$  components(with the
   exception of components consisting of $H_{p}\backslash Z'$) each of which has a unique vertex in $H_{m}\bigcup H_{n}$  of
    out-degree zero and the remaining vertices all have out-degree one; if we
    add edges from some vertices of these vertices of out-degree zero to
    $z_{i_{1}},z_{i_{2}},\cdots ,z_{i_{k}}$, we obtain a forest in $B(p,r)$  that contains $G$. There are
    $(r+1)^{l+k-t}$ ways. Therefore, this implies that there are
    \begin{equation*}
    \sum_{t=0}^{l+k-1}{l+k-1 \choose t}(p-r)^{t}(r+1)^{l+k-t}=(r+1)(p+1)^{l+k-1} \\
    \end{equation*}
   forests   in  $B(n,r)$   that contain $G$.

   Hence, 
  \begin{eqnarray*}
    S(m,n,p)
    &=&\sum_{l=0}^{m}\sum_{k=0}^{n}\sum_{r=0}^{p}f(m,l;n,k)(m+n)^{p-r}(r+1)(p+1)^{l+k-1} \\
    &=&\sum_{l=0}^{m}\sum_{k=0}^{n}\sum_{r=0}^{p}{m \choose l}{n \choose k}{p \choose r}n^{m-l-1}m^{n-k-1}\\
    &&(km+ln-lk)(m+n)^{p-r}(r+1)(p+1)^{l+k-1}\\
    &=&(m+n+1)^{p-1}(m+p+1)^{n-1}(n+p+1)^{m-1}(m+n+p+1)^{2}.
  \end{eqnarray*}
Therefore, we get the required result.
\end{proof} 

%

 \textbf{Acknowledgement} I would like to thank the editors and anonymous reviewers for their help and advices for this article.

 \end{document}